\documentclass[centertags,reqno]{pspum-l}
\usepackage{amsmath,amsthm,amscd,amssymb}
\usepackage{latexsym}
\usepackage{graphicx}
\usepackage[mathscr]{eucal}
\newcommand{\bbC}{{\mathbb{C}}}
\newcommand{\bbD}{{\mathbb{D}}}

\newcommand{\bbR}{{\mathbb{R}}}

\newcommand{\bbZ}{{\mathbb{Z}}}

\newcommand{\fre}{{\frak{e}}}

\newcommand{\bk}{{\mathbf{k}}}

\newcommand{\calE}{{\mathcal{E}}}

\newcommand{\calS}{{\mathcal S}}
\newcommand{\calT}{{\mathcal T}}

\newcommand{\lb}{\label}
\newcommand{\f}{\frac}

\newcommand{\ol}{\overline}
\newcommand{\ti}{\tilde  }
\newcommand{\wti}{\widetilde  }

\newcommand{\dist}{\text{\rm{dist}}}

\newcommand{\ess}{\text{\rm{ess}}}
\newcommand{\ac}{\text{\rm{ac}}}

\newcommand{\s}{\text{\rm{s}}}

\newcommand{\intt}{\text{\rm{int}}}

\newcommand{\bi}{\bibitem}

\newcommand{\beq}{\begin{equation}}
\newcommand{\eeq}{\end{equation}}
\newcommand{\ba}{\begin{align}}
\newcommand{\ea}{\end{align}}
\newcommand{\veps}{\varepsilon}

\newcounter{smalllist}
\newenvironment{SL}{\begin{list}{{\rm\roman{smalllist})}}{%
\setlength{\topsep}{0mm}\setlength{\parsep}{0mm}\setlength{\itemsep}{0mm}%
\setlength{\labelwidth}{2em}\setlength{\leftmargin}{2em}\usecounter{smalllist}%
}}{\end{list}}

\newcommand{\comm}[1]{}

\DeclareMathOperator{\Real}{Re}
\DeclareMathOperator{\Ima}{Im}

\allowdisplaybreaks
\numberwithin{equation}{section}

\newtheorem{theorem}{Theorem}[section]

\newtheorem*{p2.1}{Proposition 2.1}

\newtheorem{corollary}[theorem]{Corollary}
\theoremstyle{definition}
\newtheorem*{remark}{Remark}
\newtheorem*{remarks}{Remarks}

\newcommand{\abs}[1]{\lvert#1\rvert}
\newcommand{\jap}[1]{\langle #1 \rangle}

\begin{document}

\title{Finite Gap Jacobi Matrices: A Review}
\author{Jacob S.\ Christiansen}
\address{Department of Mathematical Sciences, University of Copenhagen,
Universitetsparken 5, DK-2100 Copenhagen, Denmark}
\email{stordal@math.ku.dk}
\thanks{The first author was supported in part by a Steno Research Grant (09-064947) from
the Danish Research Council for Nature and Universe}

\author{Barry Simon}
\address{Mathematics 253-37, California Institute of Technology, Pasadena, CA 91125, USA}
\email{bsimon@caltech.edu}
\thanks{The second author was supported in part by NSF grant DMS-0968856}

\author{Maxim Zinchenko}
\address{Department of Mathematics and Statistics, University of New Mexico, Albuquerque, NM 87131, USA}
\email{maxim@math.unm.edu}
\thanks{The third author was supported in part by NSF grant DMS-0965411}

\date{August 4, 2012} 


\maketitle

\section{Introduction} \lb{s1}

Perhaps the most common theme in Fritz Gesztesy's broad opus is the study of problems with
periodic or almost periodic finite gap differential and difference equations, especially
those connected to integrable systems. The present paper reviews recent progress in the understanding
of finite gap Jacobi matrices and their perturbations. We'd like to acknowledge our debt to Fritz
as a collaborator and friend. We hope Fritz enjoys this birthday bouquet!

We consider Jacobi matrices, $J$, on $\ell^2 (\{1,2,\dots,\})$ indexed by $\{a_n, b_n\}_{n=1}^\infty$,
$a_n >0$, $b_n\in\bbR$, where ($u_0\equiv 0$)
\begin{equation} \lb{1.1}
(Ju)_n = a_n u_{n+1} + b_n u_n +a_{n-1} u_{n-1}
\end{equation}
or its two-sided analog on $\ell^2(\bbZ)$ where $a_n,b_n, u_n$ are indexed by $n\in\bbZ$ and $J$ is still
given by \eqref{1.1} (we refer to ``Jacobi matrix'' for the one-sided objects and ``two-sided Jacobi
matrix'' for the $\bbZ$ analog). Here the $a$'s and $b$'s parametrize the operator $J$ and $\{u_n\}
\in\ell^2$.

We recall that associated to each bounded Jacobi matrix, $J$, there is a unique probability measure, $\mu$,
of compact support in $\bbR$ characterized by either of the equivalent
\begin{SL}
\item[(a)] $J$ is unitarily equivalent to multiplication by $x$ on $L^2 (\bbR,d\mu)$ by a unitary with
$(U\delta_1)(x)\equiv 1$.
\item[(b)] $\{a_n, b_n\}_{n=1}^\infty$ are the recursion parameters for the orthogonal polynomials for $\mu$.
\end{SL}
We'll call $\mu$ the spectral measure for $J$.

By a finite gap Jacobi matrix, we mean one whose essential spectrum is a finite union
\begin{equation} \lb{1.2}
\sigma_\ess (J) = \fre \equiv [\alpha_1,\beta_1]\cup \cdots \cup[\alpha_{\ell+1}, \beta_{\ell+1}]
\end{equation}
where
\begin{equation} \lb{1.3}
\alpha_1 < \beta_1 < \cdots < \alpha_{\ell+1} < \beta_{\ell+1}
\end{equation}
$\ell$ counts the number of gaps.

We will see that for each such $\fre$, there is an $\ell$-dimensional torus of two-sided $J$'s with
$\sigma (J) =\fre$ and $J$ almost periodic and regular in the sense of
Stahl--Totik \cite{StT}. We'll present the theory of perturbations of such $J$ that decay but not too
slowly. Our interest will be in spectral types, Lieb--Thirring bounds on the discrete eigenvalues and on
orthogonal polynomial asymptotics. We begin in Section~\ref{s2} with a discussion of the case $\ell=0$
where we may as well take $\fre = [-2,2]$, in which the ($0$-dimensional) torus is the single point with
$a_n\equiv 1$, $b_n\equiv 0$. We'll discuss the theory in that case as background.

Section~\ref{s3} describes the isospectral torus. Section~\ref{s4} discusses the results for general
finite gap sets with a mention of the special results that occur if each $[\alpha_j,\beta_j]$ has
rational harmonic measure, in which case the isospectral torus contains only periodic $J$'s.
Section~\ref{s5} discusses a method for the general finite gap case which relies on the realization
of $\bbC\cup\{\infty\}\setminus\fre$ as the quotient of the unit disk in $\bbC$ by a Fuchsian
group---a method pioneered by Peherstorfer--Sodin--Yuditskii \cite{PY,SY}, who were motivated by
earlier work of Widom \cite{Widom} and Aptekarev \cite{Apt}.

While we focus on the finite gap case, we note there are some results on general compact $\fre$'s
in $\bbR$ with various restrictive conditions on $\fre$ (e.g., Parreau--Widom). Peherstorfer--Yuditskii \cite{PY}
discuss homogeneous sets and Christiansen \cite{Chr1,Chr2} proves versions of Theorems~\ref{T4.3}
and \ref{T4.5} below for suitable infinite gap $\fre$'s. See \cite{ErYu,Y} for discussion of properties of some $\fre$'s and examples relevant to this area.

These works suggest forms of two conditions in the finite gap case suitable for generalization. Let
$\rho_\fre$ be the equilibrium measure for $\fre$ and $G_\fre(z)$ its Green's function ($-\calE
(\rho_\fre) - \Phi_{\rho_\fre}(z)$ in terms of \eqref{3.1}/\eqref{3.2}). Then \eqref{4.5} should read
\begin{equation} \lb{1.4}
\sum_{n=1}^N G_\fre (x_n) <\infty
\end{equation}
(which for finite gap $\fre$ is equivalent to \eqref{4.5}). Similarly, \eqref{4.6} should read
\begin{equation} \lb{1.5}
\int \log [f(x)]\, d\rho_\fre (x) >-\infty
\end{equation}
(again, for finite gap $\fre$ equivalent to \eqref{4.6}).

\medskip

J.S.C.\ and M.Z.\ would like to thank Caltech for its hospitality where this manuscript was written.

\section{The Zero Gap Case} \lb{s2}

The Jacobi matrix, $J_0$, with $a_n\equiv 1$, $b_n\equiv 0$ is called the free Jacobi matrix. It is
easy to see that the solutions of $J_0 u=\lambda u$ are given by solving
\begin{equation} \lb{2.1}
\alpha + \alpha^{-1} = \lambda
\end{equation}
for $\lambda\in\bbC$ and setting
\begin{equation} \lb{2.2}
u_n = \f{1}{2i}\, (\alpha^n - \alpha^{-n})
\end{equation}
This is polynomially bounded in $n$ if and only if $\abs{\alpha}=1$. If $\alpha=e^{ik}$, then
\begin{equation} \lb{2.3}
\lambda = 2\cos k, \qquad u_n = \sin (kn)
\end{equation}
Thus,
\begin{equation} \lb{2.4x}
\sigma (J_0) = [-2,2], \quad \lambda\in (-2,2) \Rightarrow\text{all eigenfunctions bounded}
\end{equation}
(by all eigenfunctions here, we mean without the boundary condition $u_0=0$).

In identifying the spectral type, the following is useful:

\begin{theorem}\lb{T2.1} Let $J$ be a Jacobi matrix with ${a_n} + {a_n^{-1}} +\abs{b_n}$ bounded.
Suppose all solutions of $(Ju)_n = \lambda u_n$ {\rm{(}}where $u_0, u_1$ are arbitrary{\rm{)}} are bounded
for $\lambda\in S\subset\bbR$. Then the spectrum of $J$ on $S$ is purely a.c.\ in the sense that if
$\mu$ is the spectral measure of $J$ and $\abs{\,\cdot\,}$ is Lebesgue measure, then
\begin{equation} \lb{2.4}
\mu_\s (S) =0, \qquad T\subset S \text{ and } \abs{T}>0\Rightarrow \mu_\ac (T) >0
\end{equation}
\end{theorem}

\begin{remark} The modern approach to this theorem would use the inequalities of Jitomirskaya--Last \cite{JL1,JL2}
or Gilbert--Pearson subordinacy theory \cite{Gil,GP,KP,Pear} to handle $\mu_\s$ and the results of Last--Simon \cite{LS}
for the a.c.\ spectrum. The simplest proof for this special case (where the above ideas are overkill) is perhaps
Simon \cite{S253}.
\end{remark}

A simple variation of parameters in the difference equation implies that under $\ell^1$ perturbations, eigenfunctions
remain bounded when $\lambda\in (-2,2)$, that is,

\begin{theorem} \lb{T2.2} Let $J$ be a Jacobi matrix with
\begin{equation} \lb{2.5}
\sum_{n=1}^\infty\, \abs{a_n-1} + \abs{b_n} <\infty
\end{equation}
Then $\sigma_\ess(J) = [-2,2]$ and the spectrum on $(-2,2)$ is purely a.c.
\end{theorem}

\begin{remark} The continuum analog of Theorem~\ref{T2.2} goes back to Titchmarsh \cite{Titch}.
\end{remark}

Thus, the spectrum outside $[-2,2]$ is a set of eigenvalues $\{x_n\}_{n=1}^N$ where $N\in\mathbb{N}\cup\{\infty\}$.
\eqref{2.5} has implications for these eigenvalues.

\begin{theorem} \lb{T2.3} Let $\{x_n\}_{n=1}^N$ be the eigenvalues of a Jacobi matrix. Then
\begin{equation} \lb{2.6}
\sum_{n=1}^N\, (x_n^2 - 4)^{1/2} \leq \sum_{n=1}^\infty\, \abs{b_n} + 4 \sum_{n=1}^\infty\, \abs{a_n-1}
\end{equation}
\end{theorem}

\begin{remarks} 1. This implies
\begin{equation} \lb{2.7}
\sum_{n=1}^N \dist (x_n, [-2,2])^{1/2} \leq \frac12 \biggl(\, \sum_{n=1}^\infty \, \abs{b_n} +
4\sum_{n=1}^\infty\, \abs{a_n-1}\biggr)
\end{equation}

\smallskip
2. The analog of \eqref{2.7} in the continuum case is due to Lieb--Thirring \cite{LT2} who proved it
when the power $1/2$ is replaced by $p>1/2$ and the right side is replaced by $\abs{b_n}^{p+1/2}$,
$\abs{a_n-1}^{p+1/2}$ and $1/2$ by a suitable constant. They proved the analog is false if $p<1/2$ and conjectured
the result if $p=1/2$. This conjecture was proven by Weidl \cite{Weidl} with an alternate proof and optimal
constant by Hundermark--Lieb--Thomas \cite{HLT}. \eqref{2.7} and its $p>1/2$ analogs are called Lieb--Thirring inequalities after \cite{LT2}.

\smallskip
3. This theorem is a result of Hundertmark--Simon \cite{HS} who used a method inspired by \cite{HLT}.

\smallskip
4. \eqref{2.6} is optimal in the sense that its $p<1/2$ analog is false and one cannot put a constant
$\gamma <1$ in front of neither the $b$ sum nor the $a-1$ sum. The same also applies to \eqref{2.7}.

\smallskip
5. \eqref{2.6} implies $p>1/2$ analogs by an argument of Aizenman--Lieb \cite{AL}.

\smallskip
6. The one-half power in \eqref{2.6}/\eqref{2.7} is especially significant for the following reason:
\begin{equation} \lb{2.8}
x(z) = z+z^{-1}
\end{equation}
maps $\bbD$ to $\bbC\cup\{\infty\}\setminus [-2,2]$. Its inverse
\begin{equation} \lb{2.9x}
z(x) = \tfrac12\, \bigl( x - \sqrt{x^2 -4}\,\bigr)
\end{equation}
has a square root singularity at $x=\pm 2$. Thus, the finiteness of the left side of \eqref{2.6}/\eqref{2.7}
is equivalent to a Blaschke condition
\begin{equation} \lb{2.10x}
\sum_{n=1}^N\, (1 - \abs{z(x_n)}) <\infty
\end{equation}
\end{remarks}

\begin{theorem}\lb{T2.4} Let $J$ be a Jacobi matrix with $\sigma_\ess (J) =[-2,2]$ and Jacobi
parameters $\{a_n,b_n\}_{n=1}^\infty$. Suppose its spectral measure has the form
\begin{equation} \lb{2.9}
d\mu = f(x)\, dx + d\mu_\s
\end{equation}
where $d\mu_\s$ is singular with respect to $dx$. Suppose that $\{x_n\}_{n=1}^N$ are its pure points
outside $[-2,2]$. Consider the three conditions:
\begin{alignat}{2}
&\text{\rm{(a)}} \qquad && \sum_{n=1}^N \dist (x_n, [-2,2])^{1/2} <\infty \lb{2.10} \\
&\text{\rm{(b)}} \qquad && \int_{-2}^2 (4-x^2)^{-1/2} \log [f(x)]\, dx > -\infty \lb{2.11} \\
&\text{\rm{(c)}} \qquad && \lim_{n\to\infty} a_1 \dots a_n \text{ exists in } (0,\infty) \lb{2.12} \\
\intertext{Then any two conditions imply the third. Moreover, in that case,}
&\text{\rm{(d)}} \qquad && \sum_{n=1}^\infty (a_n -1)^2 + b_n^2 < \infty \lb{2.13} \\
&\text{\rm{(e)}} \qquad && \lim_{K\to\infty} \sum_{n=1}^K (a_n -1) \text{ and } \lim_{K\to\infty}
\sum_{n=1}^K b_n \text{ exist} \lb{2.14}
\end{alignat}
\end{theorem}

\begin{remarks} 1. \eqref{2.10} is called a critical Lieb--Thirring inequality. \eqref{2.11} is the Szeg\H{o}
condition.

\smallskip
2. Since $f\in L^1$, the integral in \eqref{2.11} can only diverge to $-\infty$. That is, the integral
over $\log_+$ is always finite and \eqref{2.11} is equivalent to the integral converging absolutely.

\smallskip
3. By a result of Ullman \cite{Ull}, $\sigma_\ess (J) = [-2,2]$ and $f(x) >0$ for a.e.\ $x$ in $[-2,2]$
implies $\lim_{n\to\infty} (a_1 \dots a_n)^{1/n} =1$, so \eqref{2.12} can be thought of as a second
term in the asymptotics of $\f1n \log(a_1 \dots a_n)$.

\smallskip
4. Condition (c) can be thought of as three statements: $\limsup < \infty$, $\liminf > 0$, and $\limsup
= \liminf$. The full strength of (c) is not always needed. For example, (a) plus $\limsup >0$ implies (b)
and the rest of (c).

\smallskip
5. This result can be thought of as an analog of a theorem of Szeg\H{o} for OPUC \cite{Sz20} (see also \cite[Ch.\ 2]{OPUC1}). That (b) $\Rightarrow$ (c), if there are no eigenvalues, is due to Shohat \cite{Sho} and that (b) $\Leftrightarrow$ (c), if there are finitely many $x$'s, is due to Nevai \cite{Nev79}. The general (a) $+$ (b) $\Rightarrow$ (c) is due to Peherstorfer--Yuditskii \cite{PYpams}
and the essence of this theorem is from Killip--Simon \cite{KS}, although the precise theorem is from Simon--Zlato\v{s} \cite{SZ}.
\end{remarks}

\begin{corollary} \lb{C2.5} If \eqref{2.5} holds, then so does \eqref{2.11}.
\end{corollary}

\begin{proof} \eqref{2.5} implies $\prod_{n=1}^\infty a_n$ converges absolutely and, by Theorem~\ref{T2.3},
it implies \eqref{2.10}. Thus, \eqref{2.11} holds by Theorem~\ref{T2.4}.
\end{proof}

\begin{remarks} 1. This result was a conjecture of Nevai \cite{Nev92}.

\smallskip
2. It was proven by Killip--Simon \cite{KS}. It was the need to complete the proof of this that motivated
Hundertmark--Simon \cite{HS}.
\end{remarks}

There is a close connection between these conditions and asymptotics of the OPRL:

\begin{theorem} \lb{T2.6} Let $\{p_n(x)\}_{n=0}^\infty$ be the orthonormal polynomials for a Jacobi matrix,
$J$, obeying the conditions {\rm{(a)--(c)}} of Theorem~\ref{T2.4}. Then uniformly for $x$ in compact
subsets of \,$\bbC\cup\{\infty\}\setminus [-2,2]$,
\begin{equation} \lb{2.15}
\lim_{n\to\infty} \f{p_n(x)}{\big[\f12(x+\sqrt{x^2 -4}\,)\big]^n}
\end{equation}
exists and is analytic with zeros only at the $x_n$'s.
\end{theorem}

\begin{remarks} 1. When there are no $x_n$'s, this is essentially a result of Szeg\H{o} \cite{Sz20,Sz22a}.
For the general case, see Peherstorfer--Yuditskii \cite{PYpams}.

\smallskip
2. This is called Szeg\H{o} asymptotics.

\smallskip
3. The reason for the different sign in \eqref{2.9x} and \eqref{2.15} is that, as $n\to\infty$, $p_n(x)\to\infty$,
$\abs{z(x)}<1$ so $z(x)^n p_n(x)$ is bounded. The other solution of \eqref{2.8} is $z(x)^{-1}$ and it is that
solution that appears in the denominator of \eqref{2.15}.
\end{remarks}

While conditions (a)--(c) of Theorem~\ref{T2.4} are sufficient for Szeg\H{o} asymptotics, they are not
necessary:

\begin{theorem} \lb{T2.7} Let $J$ be a Jacobi matrix whose parameters obey \eqref{2.13} and \eqref{2.14}. Then
\eqref{2.15} holds on compact subsets of $\bbC\cup\{\infty\}\setminus [-2,2]$. Conversely, if \eqref{2.15}
holds uniformly on the circle $\vert x\vert=R$ for some $R>2$, then \eqref{2.13} and
\eqref{2.14} hold.
\end{theorem}

\begin{remarks} 1. This is a result of Damanik--Simon \cite{Jost1}.

\smallskip
2. There exist examples where \eqref{2.13} and \eqref{2.14} hold but both \eqref{2.10} and \eqref{2.11} fail.
\end{remarks}

\begin{theorem} \lb{T2.8} For a Jacobi matrix, $J$, with parameters $\{a_n,b_n\}_{n=1}^\infty$, spectral
measure obeying \eqref{2.9}, and discrete eigenvalues $\{x_n\}_{n=1}^N$, one has
\begin{equation} \lb{2.16}
\sum_{n=1}^\infty (a_n-1)^2 + b_n^2 <\infty
\end{equation}
if and only if
\begin{alignat}{2}
&\text{\rm{(a)}} \qquad && \sigma_\ess (J) = [-2,2] \lb{2.17} \\
&\text{\rm{(b)}} \qquad && \sum_{n=1}^N \dist (x_n, [-2,2])^{3/2} <\infty \lb{2.18} \\
&\text{\rm{(c)}} \qquad && \int_{-2}^2 (4-x^2)^{+1/2} \log [f(x)]\, dx > -\infty \lb{2.19}
\end{alignat}
\end{theorem}

\begin{remarks} 1. This theorem is due to Killip--Simon \cite{KS}. They call (a) Blumenthal--Weyl,
(b) Lieb--Thirring, and (c) quasi-Szeg\H{o}.

\smallskip
2. The continuous analog of \eqref{2.16} $\Rightarrow$ \eqref{2.18} is due to Lieb--Thirring \cite{LT2}.
\end{remarks}

\begin{theorem} \lb{T2.8A} Let $J$ be a Jacobi matrix with $\sigma_\ess(J)=[-2,2]$ and spectral
measure, $d\mu$, given by \eqref{2.9}. Suppose $f(x)>0$ for a.e.\ $x$ in $[-2,2]$. Then
\begin{equation} \lb{2.22a}
\lim_{n\to\infty} \abs{a_n-1} + \abs{b_n}=0
\end{equation}
\end{theorem}

\begin{remark} This is often called the Denisov--Rakhmanov theorem after \cite{Rakh77,Rakh83,DenPAMS}.
The result is due to Denisov. Rakhmanov had the analog for OPUC which implies the weak version of
Theorem~\ref{T2.8A}, where $\sigma_\ess(J)=[-2,2]$ is replaced by $\sigma(J)=[-2,2]$. That the result
as stated was true was a long-standing conjecture settled by Denisov.
\end{remark}

Conditions on the spectrum combined with weak conditions on the Jacobi parameters have strong consequences. For
example, the existence of $\lim_{n\to\infty} a_1\dots a_n$ clearly has no implication for the $b$'s,
but if combined with $\sigma(J)=[-2,2]$ implies, by Theorems~\ref{T2.4} and \ref{T2.8}, that
$\sum_{n=1}^\infty b_n^2 <\infty$. Similarly, one has

\begin{theorem} \lb{T2.9} Suppose $\sigma_\ess(J) = [-2,2]$ and
\begin{equation} \lb{2.23}
\lim_{n\to\infty} (a_1 \dots a_n)^{1/n} = 1
\end{equation}
Then
\begin{equation} \lb{2.24}
\lim_{N\to\infty} \f{1}{N} \sum_{n=1}^N \, (a_n-1)^2 + b_n^2 =0
\end{equation}
\end{theorem}

\begin{remarks} 1. \eqref{2.23} says that the underlying measure is regular in the sense of Ullman--Stahl--Totik;
see the discussion in Section~\ref{s3}.

\smallskip
2. This theorem is a result of Simon \cite{S319}.
\end{remarks}

\section{The Isospectral Torus} \lb{s3}

Let $\fre$ be a finite gap set with $\ell$ gaps and $\ell+1$ components, $\fre_j = [\alpha_j,\beta_j]$, $j=1,\ldots,\ell+1$. There
is associated to $\fre$ a natural $\ell$-dimensional torus, $\calT_\fre$, of almost periodic Jacobi matrices.
If $\{a_n,b_n\}_{n=-\infty}^\infty$ are almost periodic sequences, they are determined by their
values for $n\geq 1$ so we can view the elements of $\calT_\fre$ as either one- or two-sided Jacobi matrices.
There are at least three different ways to think of $\calT_\fre$:
\begin{SL}
\item[(a)] As reflectionless two-sided Jacobi matrices, $J$, with $\sigma(J)=\fre$. This is the approach of
\cite{Batch,Bulla,GesHol,Geszt,PY,Rice,SY,Teschl}.

\item[(b)] As one-sided Jacobi matrices whose $m$-functions are minimal Herglotz functions on the Riemann
surface of $\big[\prod_{j=1}^{\ell+1} (z-\alpha_j) (z-\beta_j)\big]^{1/2}$. This is the approach of \cite{CSZ1}.

\item[(c)] As two-sided almost periodic $J$ which are regular in the sense of Stahl--Totik \cite{StT} with
$\sigma(J)=\fre$. This is the approach of \cite{KrS}.
\end{SL}

\smallskip
In understanding these notions, some elementary aspects of potential theory are relevant, so we begin by
discussing them. For discussion of potential theory ideas in spectral theory, see Stahl--Totik \cite{StT}
or Simon \cite{S317}.

On our finite gap set, $\fre$, there is a unique probability measure, $\rho_\fre$, called the equilibrium
measure which minimizes
\begin{equation} \lb{3.1}
\calE(\rho) = \int \log \abs{x-y}^{-1}\, d\rho(x) d\rho(y)
\end{equation}
among all probability measures supported on $\fre$. The corresponding equilibrium potential is
\begin{equation} \lb{3.2}
\Phi_{\rho_\fre}(x) = \int \log \abs{x-y}^{-1}\, d\rho_\fre(x)
\end{equation}
The capacity, $C(\fre)$, is defined by
\begin{equation} \lb{3.3}
C(\fre) = \exp (-\calE(\rho_\fre))
\end{equation}

A Jacobi matrix with $\sigma_\ess(J) =\fre$ has
\begin{equation} \lb{3.4}
\limsup (a_1 \dots a_n)^{1/n} \leq C(\fre)
\end{equation}
$J$ is called regular if one has equality in \eqref{3.4}. We call a two-sided Jacobi matrix regular if each of
the (one-sided) Jacobi matrices
\begin{equation} \lb{3.5}
J_+ \text{ (resp.\ $J_-$) with parameters } \{a_n,b_n\}_{n=1}^\infty \text{ (resp. $\{a_{-n},b_{-n+1}\}_{n=1}^\infty)$}
\end{equation}
is regular. $\rho_\fre$ is the density of zeros for any regular $J$ with $\sigma_\ess(J) =\fre$.

The $\ell+1$ numbers $\rho_\fre ([\alpha_j,\beta_j])$, $j=1, \dots, \ell+1$, which sum to $1$ are called
the harmonic measures of the bands. We also recall that a bounded function, $\psi$, on $\bbZ$ is called almost
periodic if $\{S^k\psi\}_{k\in\bbZ}$, where $(S^k\psi)_n=\psi_{n-k}$, has compact closure in $\ell^\infty$ (see the
appendix to Section~5.13 in \cite{Rice} for more on this class). Such $\psi$'s are associated to a continuous
function, $\Psi$, on a torus of finite or countably infinite dimension so that
\begin{equation} \lb{3.6}
\psi_n = \Psi(e^{2\pi in\omega_1},e^{2\pi in\omega_2},\dots)
\end{equation}
The set of $\{n_0 + \sum_{k=1}^K n_k \omega_k \colon n_0, n_k\in\bbZ, \; \sum_{k=1}^K |n_k|<\infty\}$ is called the frequency module of $\psi$
when there is no proper submodule (over $\bbZ$) that includes all the nonvanishing Bohr--Fourier coefficients.
This set for arbitrary $\{\omega_k\}_{k=1}^K$ is called the frequency module generated by $\{\omega_k\}_{k=1}^K$.

With $J_\pm$ given by \eqref{3.5}, we define $m_\pm (z)$ for $z\in\bbC\setminus\bbR$ by
\begin{equation} \lb{3.5a}
m_\pm (z) = \jap{\delta_1, (J_\pm -z)^{-1}\delta_1}
\end{equation}
One has for a two-sided Jacobi matrix that
\begin{equation} \lb{3.5b}
\jap{\delta_0, (J-z)^{-1}\delta_0} = -(a_0^2 m_+(z) - m_-(z)^{-1})^{-1}
\end{equation}
An important fact is that $J_\pm$ are determined by $m_\pm$, essentially because $m_\pm$ determine the spectral measures $\mu_\pm$ via their Herglotz representations,
\begin{equation} \lb{3.5c}
m_\pm(z) = \int \f{d\mu_\pm(x)}{x-z}
\end{equation}
and $\mu_\pm$ determine the $a$'s and $b$'s via recursion coefficients for OPRL. Alternatively, the Jacobi parameters
can be read off a continued fraction expansion of $m_\pm(z)$ at $z=\infty$.

It is sometimes useful to let $\wti J_-$ have parameters $\{a_{-n-1}, b_{-n}\}_{n=1}^\infty$, in which
case
\begin{equation} \lb{3.5d}
\jap{\delta_0, (J-z)^{-1}\delta_0} = -(z-b_0 + a_0^2 m_+(z) + a_{-1}^2 \ti m_-(z))^{-1}
\end{equation}

We can now turn to the descriptions of the isospectral torus. A two-sided Jacobi matrix, $J$, is called
reflectionless on $\fre$ if for a.e.\ $\lambda\in\fre$ and all $n$,
\begin{equation} \lb{3.7}
\Real \jap{\delta_n, (J- (\lambda + i0))^{-1} \delta_n} =0
\end{equation}
($g(\lambda + i0)$ means $\lim_{\veps\downarrow 0}g(\lambda + i\veps)$).
It is known that this is equivalent to
\begin{equation} \lb{3.8}
a_0^2 m_+ (\lambda + i0) \, \ol{m_- (\lambda + i0)} =1 \text{ for a.e. } \lambda\in\fre
\end{equation}

\subsection*{First Definition of the Isospectral Torus} A two-sided Jacobi matrix, $J$, is said to lie in
the isospectral torus, $\calT_\fre$, for $\fre$ if $\sigma(J)=\fre$ and $J$ is reflectionless on $\fre$.

\smallskip

$G_{00}(z)=\jap{\delta_0, (J-z)^{-1} \delta_0}$ is determined by $\Ima\log(G_{00}(x+i0))$ via an exponential Herglotz
representation. This argument is $\pi/2$ on $\fre$, $0$ on $(-\infty,\alpha_1)$, and $\pi$ on
$(\beta_{\ell+1},\infty)$. $G_{00}$ is real in each gap and monotone, so $G_{00}$ has at most one zero and that zero determines $\Ima \log(G_{00}(x+i0))$ on that gap.
If $G_{00}>0$ on $(\beta_j,\alpha_{j+1})$ we'll say the zero is at $\beta_j$ and if $G_{00} < 0$ on $(\beta_j,\alpha_{j+1})$ the zero is at $\alpha_{j+1}$. Thus, the zeros of $G_{00}$ determine $G_{00}$ and so $\Ima G_{00}(\lambda +i0)$ on $\fre$.

By \eqref{3.5d}, $G_{00}$ has a zero at $\lambda_0$ if and only if $m_+$ or $\ti m_-$ has a pole at
$\lambda_0$, and one can show that $m_+$ and $\ti m_-$ have no common poles. The residue of the pole is determined by
the derivative of $G_{00}$ at $\lambda=\lambda_0$. The reflectionless condition determines $\Ima m_+$
and $\Ima \ti m_-$ on $\fre$, so $a_0,a_{-1},b_0, m_+, \ti m_-$, and thus $J$, are uniquely determined
by knowing the position of the zero and if they are in the gaps (as opposed to the edges) whether the
poles are in $m_+$ or $\ti m_-$. Hence, for each gap, we have the two copies of $(\beta_j,\alpha_{j+1})$
glued at the ends, that is, a circle. Thus, given that one can show each possibility occurs, $\calT_\fre$
is a product of $\ell$ circles, that is, a torus. It is not hard to show that the Jacobi parameters depend
continuously on the positions of the zeros of $G_{00}$ and $m_+/\ti m_-$ data.

\smallskip
We turn to the second approach. Any $G_{00}$ as above is purely imaginary on the bands which implies, by
the reflection principle, that it can be meromorphically continued to a matching copy of $\calS_+\equiv
\bbC\cup\{\infty\}\setminus\fre$. This suggests meromorphic functions on $\calS$, two copies of $\calS_+$
glued together along $\fre$, will be important. $\calS$ is precisely the compactified Riemann surface of $\sqrt{R(z)}$,
where
\begin{equation} \lb{3.9}
R(z) = \prod_{j=1}^{\ell+1} (z-\alpha_j) (z-\beta_j)
\end{equation}
$\calS$ is a Riemann surface of genus $\ell$. Meromorphic functions on the surface that are not functions
symmetric under interchange of the sheets (i.e., meromorphic on $\bbC)$ have degree at least $\ell+1$.

By a minimal meromorphic Herglotz function, we mean a meromorphic function of degree $\ell+1$ on $\calS$ that obeys
\begin{SL}
\item[(i)] $\Ima f>0$ on $\calS_+\cap\bbC_+$ ($\bbC_+ = \{z\colon \Ima z>0\}$)
\item[(ii)] $f$ has a zero at $\infty$ on $\calS_+$ and a pole at $\infty$ on $\calS_-$.
\end{SL}

\smallskip
Such functions must have their $\ell$ other poles on $\bbR$ in the gaps on one sheet or the other and
are uniquely, up to a constant, determined by these $\ell$ poles, one per gap. Each ``gap,''
when you include the two sheets and branch points at the gap edges, is a circle. So if we normalize by
$m(z) =-z^{-1} + O(z^{-2})$ near $\infty$ on $\calS_+$, the set of such minimal normalized Herglotz functions
is an $\ell$-dimensional torus. Each such Herglotz function can be written on $\calS_+\cap\bbC_+$ as
\begin{equation} \lb{3.10}
m(z) = \int \f{d\mu(x)}{x-z}
\end{equation}
where $\mu$ is supported on $\fre$ plus the poles of $m$ in the gaps on $\calS_+$. $\mu$ then
determines a Jacobi matrix.

\subsection*{Second Definition of the Isospectral Torus}The isospectral torus, $\calT_\fre$, is the
set of one-sided $J$'s whose $m$-functions are minimal Herglotz functions on the compact Riemann surface $\calS$ of $\sqrt{R}$ given by \eqref{3.9}.

\smallskip

The relation between the two definitions is that the restrictions of the two-sided $J$'s to the one-sided are
these $J$ given by minimal Herglotz functions. In the other direction, each $J$ is almost periodic
and so has a unique almost periodic two-sided extension.

\subsection*{Third Definition of the Isospectral Torus} The isospectral torus is the almost periodic two-sided
$J$'s with $\sigma(J)=\fre$ and which are regular.

\smallskip

This is equivalent to the reflectionless definition since
regularity implies the Lyapunov exponent is zero and then Kotani theory \cite{Kot,S168} implies $J$ is reflectionless.

\medskip
As noted, the $J$'s in the isospectral torus are all almost periodic. Their frequency module is generated by
the harmonic measures of the bands. In particular, the elements of the isospectral torus are periodic if and only if all harmonic measures are rational. Their spectra are purely a.c.\ and all solutions of $Ju=\lambda u$ are bounded for any $\lambda\in\fre^\intt$.

Szeg\H{o} asymptotics is more complicated than in the $\ell=0$ case. One has for the OPRL associated to a point
in the isospectral torus (thought of as a one-sided Jacobi matrix) that for all $z\in\bbC\setminus\sigma(J)$,
\begin{equation} \lb{3.11}
p_n(z)\exp (-n\Phi_{\rho_\fre}(z))
\end{equation}
is asymptotically almost periodic as a function of $n$ with magnitude bounded away from $0$ for all $n$. The
frequency module is $z$-dependent (as written, this is even true if $\ell=0$ as can bee seen from the free case): the frequencies come from
the harmonic measures of the bands plus one that comes from the conjugate harmonic function of $\Phi_{\rho_\fre}
(z)$ in $\bbC_+$ (which gives the $z$-dependence of the frequency module). The limit of \eqref{3.11} on $\fre$,
where $\Phi_{\rho_\fre}(x)=0$, yields the boundedness of solutions of $(J-\lambda) u=0$. There is also a limit
at $z=\infty$: $a_1\dots a_n/C(\fre)^n$ which is almost periodic.

\section{Results in the Finite Gap Case} \lb{s4}

As we've seen, if $\ti J$ is in the isospectral torus for $\fre$ and $\lambda\in\fre^\intt$, then all solutions
of $\ti Ju = \lambda u$ are bounded. This remains true under $\ell^1$ perturbations by a variation of parameters,
so Theorem~\ref{T2.1} is applicable and we have

\begin{theorem}\lb{T4.1} Let $\fre$ be a finite gap set and $\ti J$, with parameters $\{\ti a_n,\ti b_n\}_{n=1}^\infty$, an element of $\calT_\fre$, the isospectral torus for $\fre$. Let $J$ be a Jacobi
matrix with
\begin{equation} \lb{4.1}
\sum_{n=1}^\infty\, \abs{a_n - \ti a_n} + \abs{b_n -\ti b_n} <\infty
\end{equation}
Then $\sigma_\ess(J) = \fre$ and the spectrum on $\fre^\intt$ is purely a.c.
\end{theorem}

\begin{remark} We are not aware of this appearing explicitly in the literature, although it follows easily from
results in \cite{PY,CSZ1}.
\end{remark}

As for eigenvalues in $\bbR\setminus\fre$:

\begin{theorem} \lb{T4.2} There is a constant $C$ depending only on $\fre$ so that for any Jacobi matrix, $J$,
obeying \eqref{4.1} for some $\ti J\in\calT_\fre$, we have, with $\{x_n\}_{n=1}^N$ the eigenvalues of $J$,
\begin{equation} \lb{4.2}
\sum_{n=1}^N \dist (x_n,\fre)^{1/2} \leq C_0 + C\biggl(\, \sum_{n=1}^\infty\, \abs{a_n - \ti a_n} +
\abs{b_n - \ti b_n}\biggr)
\end{equation}
where
\begin{equation} \lb{4.3}
C_0 = \sum_{j=1}^\ell\, \bigg|\frac{\alpha_{j+1} - \beta_j}{2}\bigg|^{1/2}
\end{equation}
\end{theorem}

\begin{remarks} 1. This result is essentially in Frank--Simon \cite{FS}. They are only explicit about
perturbations of two-sided Jacobi matrices where $\ti J$ has no eigenvalues. They mention that one can use interlacing
to then get results for the one-sided case---this makes that idea explicit.

\smallskip
2. Prior to \cite{FS}, Frank--Simon--Weidl \cite{FSW} proved such a bound on the $x_n$ in $\bbR\setminus
[\alpha_1,\beta_{\ell+1}]$ and Hundertmark--Simon \cite{HS08} if $1/2$ in the power of $\dist(\ldots)^{1/2}$ is replaced by
$p>1/2$ and $1$ in the power of $\abs{a_n-\ti a_n}$ and $\abs{b_n-\ti b_n}$ by $p+1/2$, that is,
noncritical Lieb--Thirring bounds.
\end{remarks}

\begin{theorem} \lb{T4.3} Let $J$ be a Jacobi matrix with $\sigma_\ess (J)=\fre$ and Jacobi
parameters $\{a_n,b_n\}_{n=1}^\infty$. Suppose its spectral measure has the form
\begin{equation} \lb{4.4}
d\mu = f(x)\, dx + d\mu_\s
\end{equation}
where $d\mu_\s$ is singular with respect to $dx$. Suppose $\{x_n\}_{n=1}^N$ are the pure points of $d\mu$
outside $\fre$. Consider the three conditions:
\begin{alignat}{2}
&\text{\rm{(a)}} \qquad && \sum_{n=1}^N \dist (x_n,\fre)^{1/2} < \infty \lb{4.5} \\
&\text{\rm{(b)}} \qquad && \int_\fre \dist (x,\bbR\setminus\fre)^{-1/2}\log [f(x)]\,dx > -\infty \lb{4.6} \\
&\text{\rm{(c)}} \qquad && \text{For some constant $R>1$, }
R^{-1} \leq \f{a_1 \dots a_n}{C(\fre)^n} \leq R \lb{4.7}
\end{alignat}
Then any two imply the third, and if they hold, there exists $\ti J\in
\calT_\fre$ so that
\begin{equation} \lb{4.8}
\lim_{n\to\infty} \abs{a_n - \ti a_n} + \abs{b_n -\ti b_n} =0
\end{equation}
Moreover,
\begin{alignat}{2}
&\text{\rm{(d)}} \qquad && \lim_{n\to\infty} \f{a_1 \dots a_n}{\ti a_1 \dots \ti a_n} \text{ exists in }
(0,\infty) \lb{4.9} \\
&\text{\rm{(e)}} \qquad && \lim_{K\to\infty} \sum_{n=1}^K \, (b_n - \ti b_n) \text{ exists in }
\bbR \lb{4.10}
\end{alignat}
\end{theorem}

\begin{remarks} 1. Depending on which implications one looks at, only part of (c) is needed. For example,
if (a) holds,
\begin{equation} \lb{4.11}
\text{\rm{(b)}} \; \Leftrightarrow \; \limsup_{n\to\infty} \f{a_1 \dots a_n}{C(\fre)^n} > 0
\end{equation}
(that is, indeed, $\limsup$ and not $\liminf$).

\smallskip
2. As stated, this theorem (except for (e); see below) is due to Christiansen--Simon--Zinchenko
\cite{CSZ2}, but parts of it were known. While \cite{CSZ2} focus on Szeg\H{o} asymptotics (see below),
the work of Widom \cite{Widom} and Aptekarev \cite{Apt} implied if there are no or finitely many $x_n$'s,
then (b) $\Rightarrow$ (c), and Peherstorfer--Yuditskii \cite{PY} proved (a) $+$ (b) $\Rightarrow$ (c)
(and as noted to us privately by Peherstorfer, combining their results and an idea of Garnett \cite{Gar}
yields \eqref{4.11}).

\smallskip
3. That (e) holds does not seem to have been noted before, although it follows easily from the results in \cite{CSZ2}. For $g_n(z) \equiv p_n(z)/\ti p_n(z)$ has a limit as $n\to\infty$ on $\bbC\setminus[\alpha_1,\beta_{\ell+1}]$
and that limit also exists and is analytic and nonzero at infinity (see Theorem \ref{T4.5} below). Since
\begin{equation} \lb{4.12}
z^{-n} p_n(z) = (a_1 \dots a_n)^{-1} \biggl( 1-\biggl(\, \sum_{j=1}^n b_j\biggr) z^{-1} + O(z^{-2})\biggr)
\end{equation}
near $z=\infty$,
\begin{equation} \lb{4.13}
\log (g_n(z)) =-\log \biggl( \f{a_1 \dots a_n}{\ti a_1 \dots \ti a_n}\biggr) - \bigg[\, \sum_{j=1}^n
\, (b_j - \ti b_j)\biggr] z^{-1} + O(z^{-2})
\end{equation}
so convergence of the analytic functions uniformly near $\infty$ implies convergence of the $O(z^{-1})$ term.
\end{remarks}

Theorems~\ref{T4.2} and \ref{T4.3} immediately imply:

\begin{corollary} \lb{C4.4} If \eqref{4.1} holds, so does \eqref{4.6}.
\end{corollary}

\begin{proof} Since $\ti a_1 \dots \ti a_n/C(\fre)^n$ is almost periodic bounded away from $0$ and $\infty$,
and $\sum_{n=1}^\infty \abs{a_n - \ti a_n} <\infty$ and ${\ti a_n}$, ${\ti a_n}^{-1}$ bounded
imply $\sum_{n=1}^\infty \abs{1-a_n/\ti a_n}<\infty$, we have \eqref{4.9}, which implies \eqref{4.7}. By
Theorem~\ref{T4.2}, \eqref{4.1} $\Rightarrow$ \eqref{4.5}, so Theorem~\ref{T4.3} implies \eqref{4.6}.
\end{proof}

\begin{remark} This is a result of \cite{FS}, although \cite{CSZ2} conjectured Theorem~\ref{T4.2} and noted
it would imply this corollary.
\end{remark}

\begin{theorem} \lb{T4.5} If the conditions {\rm{(a)--(c)}} of Theorem~\ref{T4.3} hold, then for all
$z\in\bbC\cup\{\infty\} \setminus [\alpha_1, \beta_{\ell+1}]$, $\lim_{n\to\infty} p_n(z)/\ti p_n(z)$ exists and the limit
is analytic with zeros only at the $x_n$ in $\bbR\setminus [\alpha_1, \beta_{\ell+1}]$.
\end{theorem}

\begin{remarks} 1. In this form, this result is from \cite{CSZ2}, although earlier it appeared implicitly in Peherstorfer--Yuditskii \cite{PY,PYarx}, and special
cases (with stronger assumptions on the $x_n$'s) are in \cite{Widom,Apt}. See also \cite{Rice}.

\smallskip
2. There is also an asymptotic result on $\fre$ not pointwise but in $L^2 (d\mu)$ sense; see \cite{CSZ2}.
\smallskip

3. Asymptotics results for orthogonal polynomials on finite gap sets have been pioneered by Akhiezer and Tom\v{c}uk \cite{Ah60, AT61}.
\end{remarks}

We do not know an analog of the ``if and only if'' statement of Theorem~\ref{T2.7}, but there is one direction:

\begin{theorem} \lb{T4.6} Let $\{\ti a_n, \ti b_n\}_{n=1}^\infty$ be an element of the isospectral torus,
$\calT_\fre$, of a finite gap set, $\fre$. Let $\{a_n,b_n\}_{n=1}^\infty$ be another set of Jacobi parameters
and $\delta a_n,\delta b_n$ given by
\[
\delta a_n = a_n - \ti a_n, \qquad \delta b_n = b_n - \ti b_n
\]
Suppose that
\begin{SL}
\item[{\rm{(a)}}]
\begin{equation} \lb{4.14}
\sum_{n=1}^\infty\, \abs{\delta a_n}^2 + \abs{\delta b_n}^2 <\infty
\end{equation}
\item[{\rm{(b)}}] For any $\bk \in\bbZ^\ell$,
\begin{equation} \lb{4.15}
\sum_{n=1}^N e^{2\pi i (\bk\cdot \pmb\omega) n} \delta a_n \quad\text{and}\quad \sum_{n=1}^N
e^{2\pi i (\bk \cdot \pmb\omega) n} \delta b_n
\end{equation}
have {\rm{(}}finite{\rm{)}} limits as $N\to\infty$.
\item[{\rm{(c)}}] For every $\veps > 0$,
\begin{equation} \lb{4.16}
\sup_N \biggl\{ \biggl| \sum_{n=1}^N e^{2\pi i (\bk\cdot \pmb\omega) n} \delta a_n \biggr| +
\biggl| \sum_{n=1}^N e^{2\pi i (\bk \cdot \pmb\omega) n} \delta b_n\biggr| \biggr\} \leq C_\veps \exp (\veps \abs{\bk})
\end{equation}
\end{SL}
Let $p_n(z)$ {\rm{(}}resp.\ $\ti p_n(z)${\rm{)}} be the orthonormal polynomials for $\{a_n,b_n\}_{n=1}^\infty$
{\rm{(}}resp.\ $\{\ti a_n, \ti b_n\}_{n=1}^\infty${\rm{)}}. Then for any $z\in\bbC\setminus\bbR$,
\begin{equation} \lb{4.17}
\lim_{n\to\infty}\, \f{p_n(z)}{\ti p_n(z)}
\end{equation}
exists and is finite and nonzero.
\end{theorem}

\begin{remarks} 1. Here $\pmb\omega=(\omega_1,\dots,\omega_\ell)$ is the $\ell$-tuple of harmonic measures (i.e., $\omega_j=\rho_\fre([\alpha_j,\beta_j])$)
and $\bk\cdot \pmb\omega =\sum_{j=1}^\ell
k_j \omega_j$. We thus require infinitely many conditions.

\smallskip
2. This result is from \cite{CSZ3}.

\smallskip
3. If the torus consists of period $p$ elements (i.e., each $\rho_\fre ([\alpha_j, \beta_j])$ is $k_j/p$, where there
is no common factor for $p,k_1, \dots, k_\ell$), then the infinity of conditions \eqref{4.15} reduces to
the finitely many conditions that for $j=1,2,\dots, p$, $\sum_{n=0}^N \delta a_{np+j}$ and $\sum_{n=0}^N \delta b_{np+j}$ have finite limits and \eqref{4.16} becomes automatic.

\smallskip
4. \cite{CSZ3} uses this theorem to construct examples where Szeg\H{o} asymptotics holds, but both \eqref{4.5} and
\eqref{4.6} fail to hold.
\end{remarks}

An analog of Theorem~\ref{T2.8} is not known for general $\fre$ but is known in one special case. We say $\fre$ is
$p$-periodic with all gaps open if $\ell=p-1$, and for $j=1, \dots, p$, $\rho_\fre ([\alpha_j,\beta_j])=1/p$.

We also need a notion of approach to the isospectral torus rather than a single element. Given two Jacobi matrices,
we define
\begin{equation} \lb{4.18}
d_m (J,J') = \sum_{k=0}^\infty e^{-\abs{k}} (\abs{a_{m+k} - a'_{m+k}} + \abs{b_{m+k} - b'_{m+k}})
\end{equation}
and
\begin{equation} \lb{4.19}
d_m (J,\calT_\fre) = \inf_{J'\in\calT_\fre} d_m (J, J')
\end{equation}

\begin{theorem} \lb{T4.7} Let $\fre$ be $p$-periodic with all gaps open. Let $J$ be a Jacobi matrix with
spectral measure obeying \eqref{4.4} and eigenvalues $\{x_n\}_{n=1}^N$ outside $\fre$. Then
\begin{equation} \lb{4.20}
\sum_{m=1}^\infty d_m (J,\calT_\fre)^2 <\infty
\end{equation}
if and only if
\begin{alignat}{2}
&\text{\rm{(a)}} \qquad && \sigma_\ess (J) =\fre \lb{4.21} \\
&\text{\rm{(b)}} \qquad && \sum_{n=1}^N \dist(x_n,\fre)^{3/2} < \infty \lb{4.22} \\
&\text{\rm{(c)}} \qquad && \int_\fre \dist (x,\bbR\setminus\fre)^{+1/2} \log [f(x)] \, dx >-\infty \lb{4.23}
\end{alignat}
\end{theorem}

\begin{remark} This theorem is due to Damanik--Killip--Simon \cite{DKS}. Their method is specialized to the
periodic case, and in that case, proves some of the earlier results of this section, such as Theorem~\ref{T4.2}.
\end{remark}

\begin{theorem} \lb{T4.7A} Suppose $J$ is a Jacobi matrix with $\sigma_\ess(J)=\fre$ and so that the $f$ of
\eqref{4.4} is a.e.\ strictly positive on $\fre$. Then
\begin{equation} \lb{4.23a}
\lim_{m\to\infty} d_m (J,\calT_\fre) =0
\end{equation}
\end{theorem}

\begin{remarks} 1. This is a result of Remling \cite{Rem}. For the periodic case, it was proven earlier by
\cite{DKS}, who conjecture the result for general $\fre$.

\smallskip
2. Remling replaces \eqref{4.23a} by the assertion that every right limit of $J$ (i.e., limit point of
$\{a_{n+r}, b_{n+r}\}_{n=1}^\infty$ as $r\to\infty$) is in $\calT_\fre$. By a compactness argument,
it is easy to see that this is equivalent to \eqref{4.23a}.
\end{remarks}

\begin{theorem} \lb{T4.8} Let $\fre$ be a finite gap set and $J$ a Jacobi matrix so that
\begin{alignat}{2}
&\text{\rm{(a)}} \qquad && \sigma_\ess (J) =\fre \lb{4.24} \\
&\text{\rm{(b)}} \qquad && J \text{ is regular, i.e., }
\lim_{n\to\infty} (a_1 \dots a_n)^{1/n} = C(\fre) \lb{4.25}
\end{alignat}
Then
\begin{equation} \lb{4.26}
\lim_{M\to\infty} \f{1}{M} \sum_{m=1}^M d_m (J,\calT_\fre)^2 =0
\end{equation}
\end{theorem}

\begin{remarks} 1. This result was proven in case all harmonic measures are rational by Simon \cite{S319},
who conjectured the result in general. It was proven by Kr\"uger \cite{Kru}.

\smallskip
2. By the Schwarz inequality, \eqref{4.26} is equivalent to
\begin{equation} \lb{4.27}
\lim_{M\to\infty} \f{1}{M} \sum_{m=1}^M d_m (J,\calT_\fre) =0
\end{equation}
\end{remarks}

\smallskip

We close this section on results with a list of some open questions:
\smallskip
\begin{SL}
\item[(1)] Do (a)--(c) of Theorem~\ref{T4.3} imply that
\begin{equation} \lb{4.28}
\sum_{n=1}^\infty \, (a_n -\ti a_n)^2 + (b_n - \ti b_n)^2 <\infty
\end{equation}
as is true in the case $\fre = [-2,2]$?

\item[(2)] Is there an extension of Theorem~\ref{T4.7} to the general $\fre$ case?

\item[(3)] Is there a converse to Theorem~\ref{T4.6}? This would be interesting even in the periodic case.
\end{SL}

\section{Methods} \lb{s5}

The theory of regular Jacobi matrices says one expects the leading growth of $P_n(z)$ as $n\to\infty$ to be
$\exp(n\Phi_{\rho_\fre}(z))$. $\Phi_{\rho_\fre}$ is harmonic on $\bbC\cup\{\infty\}\setminus\fre$ so we can locally define
a harmonic conjugate and so $\wti\Phi_{\rho_\fre} (z)$ analytic with $\Real\wti\Phi_{\rho_\fre} = \Phi_{\rho_\fre}$. If you circle
around $x$, $\log (z-x)$ changes by $2\pi i$, so circling around the band $[\alpha_j,\beta_j]$, we expect $\int
\log (z-x)\, d\rho_\fre(x)$ to change by $2\pi i \rho_\fre ([\alpha_j,\beta_j])$ and $\exp (-\wti\Phi_{\rho_\fre}(z))$
to have a change of phase by $\exp (-2\pi i \rho_\fre ([\alpha_j,\beta_j]))$. Thus, we are led to consider
analytic functions on $\bbC_+$ which we can continue along any curve in $\bbC\cup\{\infty\}\setminus\fre$.

To get a single-valued function, we need to lift to the universal covering space of $\bbC\cup\{\infty\}
\setminus\fre$ and $\exp (-\wti\Phi_{\rho_\fre}(z))$ will transform under the homotopy group via a character of
this group.

So long as $\ell\neq 0$, this cover, as a Riemann surface, is the disk, $\bbD$, and the deck transformations act
as a family of fractional linear transformations on the disk, that is, a Fuchsian group. The use of these
Fuchsian groups is thus critical to the theory and used to prove several of the theorems of Section~\ref{s4}
(Theorems~\ref{T4.7}, \ref{T4.7A}, and \ref{T4.8} are exceptions).

For more on Fuchsian groups, see Beardon \cite{Beardon}, Ford \cite{Ford}, Katok \cite{Katok}, Simon \cite{Rice},
and Tsuji \cite{Tsu}. The pioneers in this approach were Sodin--Yuditskii \cite{SY}. See \cite{PY,CSZ1,CSZ2,CSZ3,Rice}
for applications of these techniques.

\bigskip


\bibliographystyle{amsalpha}

\end{document}